%% file: Gudula.tex
\documentclass[12pt]{article} 

\usepackage[utf8]{inputenc}   
\usepackage[T1]{fontenc}      

\usepackage[largesc,osf]{newtxtext}
\usepackage[amsthm,nosymbolsc,vvarbb]{newtxmath}
\usepackage[cal=euler,frak=euler]{mathalpha} 

\usepackage{array,cite,enumitem,mathtools,titlesec}

\usepackage{tikz-cd}
\usetikzlibrary{babel}

\usepackage[margin=3cm]{geometry} 

\title{Projectable reduced $f$-rings\\
       admitting elimination of quantifiers}

\author{J. I. Guier \\[9pt]
{\small
Centro de Investigación en Matemática Pura y Aplicada,} \\
{\small 
Escuela de Matemática, Universidad de Costa Rica,} \\
{\small 
11501 San José, Costa Rica}}

\date{April 30, 2026}  

\input{fonts-charlas7} 

\theoremstyle{plain}
\newtheorem{them}{Theorem}[section]
\newtheorem{prop}[them]{Proposition}
\newtheorem{lem}[them]{Lemma}

\newtheorem{claim}{Claim}  

\theoremstyle{definition}

\theoremstyle{remark}
\newtheorem{rem}[them]{Remark}

\newcommand{\hideqed}{\renewcommand{\qed}{}} 

\setlist[itemize]{noitemsep}
\setlist[enumerate]{label={\textup{(\alph*)}}} 

\titleformat{\section}{\normalfont\large\bfseries}
                      {\thesection}{1em}{}
\titlespacing{\section}{0pt}{*3.5}{*2.3}
\titleformat{\subsection}{\normalfont\normalsize\bfseries}
                         {\thesubsection}{0.7em}{}
\titlespacing{\subsection}{0pt}{*3.25}{*1.5}
\titleformat{\subsubsection}{\normalfont\small\bfseries}
                            {\thesubsubsection}{0.5em}{}
\titlespacing{\subsubsection}{0pt}{*3.0}{*1.2}
\titleformat{\paragraph}[runin]{\normalfont\bfseries}{}{0pt}{}
\titlespacing{\paragraph}{0pt}{\medskipamount}{\wordsep}

\begin{document}

\maketitle

\begin{abstract}
In this note, we give a characterization of all projectable and
divisible-projectable reduced $f$-rings satisfying the first convexity
property and admitting elimination of quantifiers, in the language of
lattice-ordered rings with the divisibility relation, the radical
relation associated to the minimal prime spectrum, and the local
divisibility relation.
\end{abstract}

\section{Introduction} 
\label{sec:intro}

To place in perspective the results proved in this paper, we start by
stating them amongst the known quantifier elimination (q.e.)~results
within the realm of real closed rings, cf.~\cite{Sch-DDG, Sch-Lum}.

All rings in this paper will be commutative with unity.
Let $\Lor = \{0, 1, +, -, \cdot, <\}$ denote the language of ordered
rings and $\Lanre = \{0, 1, +, -, \cdot, \w\}$ the language of
lattice-ordered rings. The following theories of rings admit q.e.\
in the indicated language.  

\begin{enumerate}[nosep]

\item 
The theory of real closed fields (RCF) in the language $\Lor$,
cf.~\cite{Tarski1, Tarski2}.

\item 
The theory of real closed valuation rings (RCVR) in the language
$\Lor$ with an added binary relation symbol denoting the (standard)
divisibility, cf.~\cite[\S\,2]{Ch-Dick2}.

\item 
The theory of von Neumann regular real closed rings without nonzero
minimal idempotents, in the language $\Lanre$ with an extra binary
relation symbol denoting the radical relation associated to the
minimal prime spectrum, cf.~\cite{Pre-Schw}.

\end{enumerate}

More recently:
\begin{enumerate}[nosep,resume]

\item 
The theory of projectable and divisible-projectable real closed rings
satisfying a regularity condition (sc-regularity), the first convexity
property and without nonzero minimal idempotents, in the language of
$\Lanre$ with the (standard) divisibility relation, the radical
relation used in the previous item~(c), and an extra binary relation
symbol for a local divisibility relation, cf.~\cite{Guier4}.

\end{enumerate}

In Section~\ref{sec:products} below we show, suitably adapting the
technique developed in~\cite{Guier4}, that the following theories of
rings also admit~q.e.:

\begin{enumerate}[nosep,resume]

\item 
The theory of products of two real closed fields, in the same language
of item~(c). This first-order theory is given by von~Neumann regular
real closed rings which have only four idempotents.

\item 
The theory of products of two real closed valuation rings, in the
language of item~(d). This first-order theory is given by the theory
of projectable real closed rings satisfying the sc-regularity
condition and the first convexity property, again with only four
idempotents.

\end{enumerate}

In the final Section~\ref{sec:punchline}, using one of the author's
previous results \cite[Thm.~4.10]{Guier1} (the `only~if' implication),
based on earlier work of F.~Point \cite{Point, Point-vd}, we establish
that the six listed cases are the only theories of projectable and
divisible-projectable reduced $f$-rings with the first convexity
property admitting q.e.\ in the language described in cases~(d)
and~(f).

\section{Basic notions and notations} 
\label{sec:basics}

Throughout this note $A$ will be a projectable reduced $f$-ring, see
\cite[\S\,7.5]{BKW} or \cite[\S\,2]{Guier1} for definitions. By
\cite[\S\,6.12]{Kei},
\begin{equation}
A \in \Ga^a_{\Lor} \bigl( \pi A, (A/p)_{p\in\pi A} \bigl),
\label{eq:class-of-A} 
\end{equation}
where 
$$
\pi A = \set{p \in \spec(A) : p \text{ is a minimal prime ideal}}
= \specm(A).
$$
The notation $\Ga^a_\Le \bigl( X, (A_x)_{x\in X} \bigr)$ stands for
the class of (atomic) Boolean products of the family
$\set{A_x : x \in X}$ over the topological (Boolean) space $X$ in the
language~$\Le$, i.e., continuous sections from $X$ to the (disjoint)
union of the stalks $\set{A_x : x \in X}$. See
\cite[\S\,1]{Burris-Werner} or \cite[Def.~2.1]{Guier1} for these
precise definitions.

\medskip

According to \cite[\S\,1.1]{Ch-Dick2}, a \textbf{real closed valuation
ring} is an ordered domain that satisfies the intermediate value
property for polynomials in one variable, but is not a field. In
\cite[Thm.~4A and \S\,2]{Ch-Dick2}, the authors proved completeness of
this theory (RCVR) and its elimination of quantifiers in the language
$\Lor \cup \{\mid\}$ of ordered rings with the divisibility relation.

\medskip

A prime ideal $p$ of a commutative ring $A$ is minimal if and only if,
for each $x \in p$, there is some $y \notin p$ such that $xy$ is
nilpotent (equal to $0$ if the ring $A$ is reduced);
cf.~\cite[\S\,1.1]{He-Je}.

\medskip

Radical relations were introduced in \cite[\S\,1]{Pre-Schmid}, and
reintroduced -- in the opposite direction -- in
\cite[\S\,4]{Pre-Schw}. The following is the definition given
in~\cite[\S\,4]{Pre-Schw}:
\begin{align}
(1) &\enspace a \rad a,
\nonumber \\
(2) &\mot{if} a \rad b \mot{and} b \rad c \mot{then} a \rad c,
\nonumber \\
(3) &\mot{if} a \rad c \mot{and} b \rad c \mot{then} a + b \rad c,
\nonumber \\
(4) &\mot{if} a \rad b \mot{then} ac \rad bc,
\nonumber \\
(5) &\enspace a \rad 1 \mot{and} 1 \notrad 0,
\nonumber \\
(6) &\enspace b \rad b^2,
\label{eq:radical-relations} 
\end{align}
for all $a,b,c \in A$.

\goodbreak 

In this context, \cite[Thm.~2.5]{Pre-Schmid} then shows that 
for any radical relation~$\rad$, there exists a subset
$X \subseteq \spec(A)$ such that
\begin{equation}
a \rad b \iff \forall p \in X\, (a \notin p \implies b \notin p).
\label{eq:radical-relation-X} 
\end{equation}
This radical relation is denoted by~$\rad_X$. The radical relation
associated to the minimal prime spectrum is given by
\cite[Prop.~(a), p.~21]{Pre-Schw}: 
\begin{align}
a \rad_{\pi A} b &\iff  \Ann(b) \subseteq \Ann(a)
\nonumber \\
&\iff  \forall x (bx = 0 \to ax = 0)
\nonumber \\
&\iff  \forall x (ax \neq 0 \to bx \neq 0).
\label{eq:radical-relation-pi-A} 
\end{align}
From now on, this radical relation $\rad_{\pi A}$ will be denoted
simply by~$\rad$.

\medskip

In \cite[Def.~2.5]{Guier1}, a lattice ordered ring $A$ is called
\textbf{divisible-projectable} if:
$$
\forall x \forall y \Bigl( y \neq 0 \to \exists z \exists w \bigl(
x = z + w \ett z \perp w \ett y \mid
z \ett \forall w'( w' \neq 0 \ett w' \perp (w - w')
\to y \nmid w') \bigr) \Bigr)
$$
is valid in~$A$. In \cite[Def.~2.8]{Guier1}, a ring $A$ is called
\textbf{sc-regular} if there exists an element $u \in A$ such that
$\Ann(u) = \{0\}$ (or $1 \rad u$) and $u \nmid e$ for every nonzero
idempotent $e \in A$.

\medskip

By \cite[Def.~4.3]{Guier1}, the binary function symbol 
$\div(\cdot,\cdot)$ is defined, in the theory of reduced projectable
and divisible-projectable $f$-rings, by:
\begin{align}
\div(x,y) = c \longequi c \in \bipol{y} \land \exists z \exists w 
&\bigl( x = z + w \land z \perp w \land cy = z
\nonumber \\
&\quad \land \forall w' (w' \neq 0 \land w' \perp (w - w')
\to y \nmid w') \bigr).
\label{eq:div-rigmarole} 
\end{align}
In the theory of totally ordered domains, \cite[Lem.~4.4]{Guier1} 
yields:
\begin{equation}
\div(x,y) = \begin{cases}
c & \text{if } y \neq 0 \land yc = x, \\
0 & \text{if } y = 0 \lor (y \neq 0 \land y \nmid x). \end{cases}
\label{eq:div-values} 
\end{equation}
As already noted in~\eqref{eq:class-of-A}, if $A$ is a reduced
projectable $f$-ring, then by \cite[\S\,6.13]{Kei}, 
\begin{subequations}
\label{eq:three-faces-of-A} 
\begin{equation}
A \in \Ga^a_{\Lor} \bigl( X, (A_x)_{x \in X} \bigl),
\label{eq:three-faces-of-A-one} 
\end{equation}
where $X$ is a Boolean space and $\set{A_x : x \in X}$ is a family of
totally ordered domains. If in addition $A$ is divisible-projectable,
then by \cite[Prop.~2.6]{Guier1},
\begin{equation}
A \in \Ga^a_{\Lor\cup\{\mid\}} \bigl( X, (A_x)_{x\in X} \bigl),
\label{eq:three-faces-of-A-two} 
\end{equation}
where the Boolean product is considered with the divisibility relation
in the language. In this context, it easily seen that
\begin{equation}
\div_A(a,b)(x) = \div_{A_x}(a(x), b(x))
\label{eq:three-faces-of-A-three} 
\end{equation}
\end{subequations}
for all $a,b \in A$ and $x \in X$.

\medskip

Local divisibility is defined in \cite[Def.~2.2]{Guier4} by:
\begin{equation}
y \locdiv w \longequi 
w = 0 \lor \exists w' (w' \neq 0 \land w'(w - w') = 0 \land y \mid w').
\label{eq:local-divisibility} 
\end{equation}
According to \cite[end of \S\,2]{Guier4}, the sc-regularity of $A$ may
be restated by:
\begin{equation}
A \models \exists u\, \bigl( 1 \rad u \land u \notdivloc 1 \bigr).
\label{eq:sc-regularity} 
\end{equation}
Thm.~6.13 of~\cite{Guier4} established elimination of quantifiers of
the theory of projectable and divisible-projectable sc-regular real
closed rings satisfying the first convexity property without nonzero
minimal idempotents, in the language 
$\Lanre \cup \{\mid,\rad,\locdiv\}$.

\section{Products of two real closed valuation rings} 
\label{sec:products}

Let $A = V_0 \x V_1$ where $V_0$ and $V_1$ are two real closed
valuation rings. Using Prop.~3.4(ii) and Cor.~2.11 of~\cite{Guier1}
and \cite[Lem.~2.13]{GuierThese}, it may be seen that $\Th(A)$ is the
theory of sc-regular projectable real closed rings with the first
convexity property where the Boolean algebra of idempotents satisfies
$\exists_{=4}$, i.e., the formula declaring that the structure has
$4$~elements. (Note that the representation space of~$A$ is the
discrete space $\{0,1\}$, and therefore the Boolean algebra is
$2^{\{0,1\}} = 2^2$.) In this context, the divisible-projectability
property is redundant since the space is discrete: see
\cite[Prop.~2.6]{Guier1}. We denote this theory by $\rT_2 = \Th(A)$.

\begin{them} 
\label{th:T2-is-qe}
$\rT_2$ admits elimination of quantifiers in the language
$\Lanre \cup \{\mid,\rad,\locdiv\}$.
\end{them}

The proof consists of several steps.

\begin{lem} 
\label{lm:T2-is-model-complete}
The theory $\rT_2$ is model complete. 
\end{lem}

\begin{proof}
Let $A,B \models \rT_2$ be such that $A \subseteq B$ in the language
$\Lanre \cup \{\mid,\rad,\divloc\}$. Therefore $A = V_0 \x V_1$ and
$B = B_0 \x B_1$, where the $B_0$ and $B_1$ are also real closed
valuation rings. As the Boolean algebra of idempotents of~$A$ consists
of four elements $(0,0), (0,1), (1,0), (1,1)$, and that of~$B$
also has four elements, they either coincide or correspond via
$(0,1) \in A \mapsto (1,0) \in B$ and 
$(1,0) \in A \mapsto (0,1) \in B$. If they coincide, then 
$V_0 \subseteq B_0$ and $V_1 \subseteq B_1$; otherwise, one gets
$V_0 \subseteq B_1$ and $V_1 \subseteq B_0$. Hence, we may as well
suppose that
\begin{equation}
V_0 \subseteq B_0 \word{and} V_1 \subseteq B_1.
\label{eq:nice-inclusions} 
\end{equation}
Since the inclusion $A \subseteq B$ preserves the order, the radical
relation and the divisibility, the inclusions in
\eqref{eq:nice-inclusions} also preserve the order, the divisibility
and the radical relation (the last of these is trivial since $V_0$,
$V_1$, $B_0$, $B_1$ are domains). By model completeness of the theory
of real closed valuation rings \cite[Thm.~4B]{Ch-Dick2}, one deduces
$V_0 \selem B_0$ and $V_1 \selem B_1$. By the Feferman-Vaught theorem
\cite{Fe-Vau}, $A = V_0 \x V_1 \selem B_0 \x B_1 = B$, which yields 
the model completeness of~$\rT_2$.
\end{proof}

\begin{lem} 
\label{lm:T2-amalgamates}
The universal theory of $\rT_2$ has the amalgamation property.
\end{lem}

\begin{proof}
To obtain the amalgamation property for $(\rT_2)_\forall$, we
consider $A,B \models \rT_2$ and $D$ a substructure of $A$ and~$B$
admitting monomorphisms $f\colon D \to A$ and $g\colon D \to B$ in the
language $\Lanre \cup \{\mid,\rad,\divloc\}$. Therefore the Boolean
algebra of idempotents of $D$ has either four or two elements. 
Two cases ensue.

\paragraph{Case 1}
\textit{$D$ is not a domain}.

Write $A = V_1 \x V_2$ where $V_1,V_2$ are real closed valuation
rings. Set $p_1 := \set{(a,0) : a \in V_1}$, a prime ideal of~$A$,
which is clearly minimal, since $(a,0) (0,1) = (0,0)$ with
$(0,1) \notin p_1$. Similarly $p_2 := \set{(0,b) : b \in V_2}$ is
another minimal prime ideal of $A$. Since $A \models \rT_2$, the Stone
space of the Boolean algebra of idempotents of $A$ is the discrete
space of two elements; therefore:
$$
\specm (A) = \{ p_1, p_2 \}.
$$
It is clear that $f^{-1}(p_1)$ and $f^{-1}(p_2)$ are two prime ideals
of~$D$; then
$$
f^{-1} \colon \specm(V_1 \x V_2) \to \spec(D).
$$

\begin{claim} 
\label{cl:minimal-spectrum}
The minimal prime spectrum of $D$ is 
$\specm(D) = \Set{f^{-1}(p_1), f^{-1}(p_2)}$.
\end{claim}

\begin{proof}[Proof of Claim~\ref{cl:minimal-spectrum}]
Let us first see that $f^{-1}(p_1)$ is minimal. Now,
$f^{-1}(p_1) \neq \{0\}$ and $f^{-1}(p_2) \neq \{0\}$ since $D$ is not
an integral domain. So there exist nonzero $x_0 \in f^{-1}(p_1)$ and
$y_0 \in f^{-1}(p_2)$.

Denote $f_1 := \pi_1 \circ f \colon A \to V_1$ and 
$f_2 := \pi_2 \circ f \colon A \to V_2$; these are ring homomorphisms.
In particular, $f(x) = \bigl( f_1(x), f_2(x) \bigr)$ for $x \in A$.
Then $f_2(x) = 0$ and $f_1(y) = 0$, for $x \in f^{-1}(p_1)$,
$y \in f^{-1}(p_2)$. Therefore 
$f(xy) = f(x) f(y) = (f_1(x),0) (0,f_2(y)) = (0,0) = 0$; and then
$xy = 0$ since $f$ is injective. 

Note that $f_1(x_0) \neq 0$ since $x_0 \neq 0$, and  
$f_2(y_0) \neq 0$ since $y_0 \neq 0$. Then for any 
$x \in f^{-1}(p_1)$, there is some $y_0 \notin f^{-1}(p_1)$ such that
$xy_0 = 0$. So the ideal $f^{-1} (p_1)$ is a minimal prime ideal
of~$D$. Similarly, for $y \in f^{-1} (p_2)$ there is
$x_0 \notin f^{-1}(p_2)$ such that $x_0 y = 0$; and so $f^{-1}(p_2)$
is also a minimal prime ideal of~$D$. Hence,
\begin{equation}
\bigl\{ f^{-1}(p_1), f^{-1}(p_2) \bigr\} \subseteq \specm(D).
\label{eq:two-in-specmin} 
\end{equation}
To prove the reverse inclusion, take $q \in \specm(D)$. Since $D$ is
not an integral domain, there is $z \in q$, $z \neq 0$; and since $q$
is minimal, there is some $d \in D \setminus q$ such that $zd = 0$.
If $f(z) \notin p_1$ and $f(z) \notin p_2$, then $f_1(z) \neq 0$,
$f_2(z) \neq 0$. Thus $(1,1) \rad (f_1(z), f_2(z))$, that is,
$f(1) \rad f(z)$. That entails $1 \rad z$ since $f$ preserves the
radical relation; and so $\Ann(z) \subseteq \Ann(1) = \{0\}$, using
the equivalences~\eqref{eq:radical-relation-pi-A}. But that would give
$d = 0$, contradicting $d \notin q$. Thus either $f_1(z) = 0$ or
$f_2(z) = 0$, whereby $f(z) \in p_2$ or $f(z) \in p_1$.

Assume $f(z) \in p_1$; and suppose that there is some nonzero $z' \in
q$ such that $f(z') \notin p_1$, entailing $f_2(z') \neq 0$. Now,
$f(z) \in p_1$ makes $f_2(z) = 0$; and so
$$ 
f(z + z') = \bigl( f_1(z) + f_1(z'), f_2(z) + f_2(z') \bigr)
= \bigl( f_1(z) + f_1(z'), f_2(z') \bigr).
$$
We may suppose $f_1(z) + f_1(z') \neq 0$, otherwise 
$f_1(z) = -f_1(z') = f_1(-z')$; so it suffices to consider instead a
nonzero $-z' \in q$ 
with $f(-z') \notin p_1$. In that case, $f_1(z) + f_1(z') 
\neq 0$ and $f_2(z') \neq 0$.  Then 
$(1,1) \rad \bigl( f_1(z) + f_1(z'), f_2(z') \bigr)$, that is,
$$
f(1) \rad f(z + z'),
$$
and therefore $1 \rad z + z'$. Since $z + z' \in q$ and $q$ is
minimal, there is an $l \in D \setminus q$ such that $(z + z')l = 0$.
As $\Ann(z + z') \subseteq \Ann(1) = \{0\}$, again
by~\eqref{eq:radical-relation-pi-A}, it follows that $l = 0 \notin q$,
a contradiction. In fine, $f(z') \in p_1$, i.e., $z' \in f^{-1}(p_1)$,
for all $z' \in q$. Hence $q \subseteq f^{-1}(p_1)$. Indeed,
$q = f^{-1}(p_1)$ since $f^{-1}(p_1)$ is minimal.

The other possibility $f(z) \in p_2$ likewise carries us to
$q = f^{-1}(p_2)$. Therefore, the inclusion \eqref{eq:two-in-specmin}
is an equality; which proves Claim~\ref{cl:minimal-spectrum}.
\end{proof}

In this way, $f^{-1} \colon \specm(V_1 \x V_2) = \{p_1, p_2\} 
\to \bigl\{ f^{-1}(p_1), f^{-1}(p_2) \bigr\} = \specm(D)$ is
surjective.%
\footnote{The surjectivity of $f^{-1}$ may also be proved using
Propositions (a) and~(b) in \cite[p.~22]{Pre-Schw}.}

\medskip

Let us reconsider $f \colon D \to V_1 \x V_2
: x \mapsto \bigl( f_1(x), f_2(x) \bigr)$, where $f_1$ and $f_2$ are
homomorphisms of ordered rings. Since
$$
f^{-1}(p_1) = \set{d \in D : f(d) \in p_1}
= \set{d \in D : f_2(d) = 0} = \ker(f_2),
$$ 
it follows that, as rings:
\begin{equation}
D/f^{-1}(p_1) = D/\ker(f_2) \cong \im(f_2) \subseteq V_2.
\label{eq:quotient-of-D} 
\end{equation}

\begin{claim} 
\label{cl:substructure-in-Lor}
$D/f^{-1}(p_1)$ is (isomorphic to) a substructure of~$V_2$ for the
language~$\Lor$. 
\end{claim}

\begin{proof}[Proof of Claim~\ref{cl:substructure-in-Lor}]
We first show that the ideal $f^{-1}(p_1)$ is convex. 

Let $c,d \in D$ be such that $0 \leq c \leq d$ with
$d \in f^{-1}(p_1)$. Then $f(d) \in p_1$, i.e., $f_2(d) = 0$. Since
$f$ is order-preserving, $0 \leq f(c) \leq f(d)$, that is,
$(0,0) \leq \bigl( f_1(c), f_2(c) \bigr) 
\leq \bigl( f_1(d), f_2(d) \bigr)$. Thus,
$$
0 \leq f_1(c) \leq f_1(d) \word{and} 0 \leq f_2(c) \leq f_2(d) = 0.
$$
So $f_2(c) = 0$, and thus $c \in f^{-1}(p_1)$. Remark that
\eqref{eq:quotient-of-D} tells us that $D/f^{-1}(p_1)$ may be viewed
as a subring of $V_2$ by 
$d/f^{-1}(p_1) = d/\ker(f_2) \mapsto f_2(d) \in V_2$. Since
$f^{-1}(p_1)$ is a convex (minimal prime) ideal, $D/f^{-1}(p_1)$ is a
totally ordered domain where the order is given by:
$d_1/f^{-1}(p_1) \leq d_2/f^{-1}(p_1)$ if and only if there is some
$c \in f^{-1}(p_1)$ such that $d_1 \leq c + d_2$, see \S\,2.3 
and~\S\,8.3 of~\cite{BKW} for details. The injection
$D/f^{-1}(p_1) \hookto V_2$ preserves the order, for if
$d_1,d_2 \in D$ make $d_1/f^{-1}(p_1) \leq d_2/f^{-1}(p_1)$, then
there is $c \in f^{-1}(p_1)$ such that $d_1 \leq c + d_2$. Then
$f(d_1) \leq f(c + d_2) = f(c) + f(d_2)$. In short,
$$
\bigl( f_1(d_1), f_2(d_1) \bigr) 
\leq \bigl( f_1(c), f_2(c) \bigr) + \bigl( f_1(d_2), f_2(d_2) \bigr).
$$
Now $c \in f^{-1}(p_1)$ yields $f(c) \in p_1$ and so $f_2(c) = 0$.
Then
$$
\bigl( f_1(d_1), f_2(d_1) \bigr) 
\leq \bigl( f_1(c), 0 \bigr) + \bigl( f_1(d_2), f_2(d_2) \bigr).
$$
The second coordinate reads $f_2(d_1) \leq f_2(d_2)$. Since the orders
in $D/f^{-1}(p_1)$ and $V_2$ are total, the reverse implication holds.
Thus,
$$
d_1/f^{-1}(p_1) \leq d_2/f^{-1}(p_1) \iff f_2(d_1) \leq f_2(d_2).
$$
Then $D/f^{-1}(p_1)$ is an order subring of~$V_2$, hence a
substructure in the language~$\Lor$.
\end{proof}

\begin{claim} 
\label{cl:substructure-in-Lor-plus}
$D/f^{-1}(p_1)$ is (isomorphic to) a substructure of~$V_2$ for the
language $\Lor \cup \{\mid\}$.
\end{claim}

\begin{proof}[Proof of Claim~\ref{cl:substructure-in-Lor-plus}]
This inclusion preserves the divisibility relation because $f$
preserves the local divisibility. To see that, we must prove:
\begin{equation}
d_1 + f^{-1}(p_1) \bigm| d_2 + f^{-1}(p_1)
\iff f_2(d_1) \mid f_2(d_2).
\label{eq:divisibility} 
\end{equation}
Firstly, take $d_1,d_2 \in D$ such that
$d_1 + f^{-1}(p_1) \bigm| d_2 + f^{-1}(p_1)$ in $D/f^{-1}(p_1)$.
Then there is some $d \in D$ such that 
$d_1 d = d_2 \bmod f^{-1}(p_1)$. Therefore 
$d_1 d - d_2 \in f^{-1}(p_1)$, and so $f_2(d_1 d - d_2) = 0$; thus
$f_2(d_1) f_2(d) - f_2(d_2) = 0$ so that $f_2(d_1) f_2(d) = f_2(d_2)$.
Therefore, $f_2(d_1) \mid f_2(d_2)$ in $\im(f_2) \subseteq V_2$.

Secondly, choose $d_1,d_2 \in D$ such that $f_2(d_1) \mid f_2(d_2)$
in~$V_2$. To reach the left-hand side of~\eqref{eq:divisibility} in
$D/ f^{-1}(p_1)$, note that it clearly holds if $f_2(d_2) = 0$, since
$d_2 \in f^{-1}(p_1)$ and then $d_2 + f^{-1}(p_1) = 0$.

If $f_2(d_2) \neq 0$, take $a \in V_2$ such that 
$f_2(d_1) \cdot a = f_2(d_2)$; necessarily, $a \neq 0$. Now consider
$w := (0, f_2(d_2)) \neq (0,0)$ and notice that:
\begin{align*}
w (w - f(d_2)) &= (0, f_2(d_2)) 
\cdot \bigl[ (0, f_2(d_2)) - (f_1(d_2), f_2(d_2)) \bigr]
\\
&= (0, f_2(d_2)) \cdot (-f_1(d_2), 0) = (0,0), 
\end{align*}
and
$$
f(d_1) \cdot (0,a) = (f_1(d_1), f_2(d_1)) \cdot (0,a)
= \bigl( f_1(d_1) \cdot 0, f_2(d_1) \cdot a \bigr) 
= (0, f_2(d_2)) = w,
$$
whereby $f(d_1) \mid w$ in $A$. We have shown that
$$
\exists w\, \bigl( w \neq 0 \land w \cdot (w - f(d_2)) = 0 
\land f(d_1) \mid w \bigr)
$$
is true in~$A$. Then $f(d_1) \divloc f(d_2)$ in $A$; 
see~\eqref{eq:local-divisibility} above. Since $f$ preserves the
local divisibility, we obtain $d_1 \divloc d_2$ in $D$. Under the
assumption $f_2(d_2) \neq 0$, $d_2 \neq 0$ holds since $f$ is
injective. There is $d'_2 \in D$ such that $d'_2 \neq 0$,
$d_2(d_2 - d'_2) = 0$ and $d_1 \mid d'_2$ in~$D$. Then there is
some $c \in D$ with $d_1 c = d'_2$, and thus
$f_2(d_1) f_2(c) = f_2(d'_2)$. Evaluating $f_2$ on 
$d_2(d_2 - d'_2) = 0$, one finds
$f_2(d_2) \bigl( f_2(d_2) - f_2(d'_2) \bigr) = 0$. Since
$f_2(d_2) \neq 0$, $f_2(d_2) = f_2(d'_2)$ follows. That implies
$f_2(d_1) f_2(c) = f_2(d_2)$, and so $f_2(d_1 c - d_2) = 0$ and
$d_1c - d_2 \in f^{-1}(p_1)$. In summary:
$$
\bigl( d_1 + f^{-1}(p_1) \bigr) \cdot \bigl( c + f^{-1}(p_1) \bigr)
= \bigl( d_2 + f^{-1}(p_1) \bigr).
$$
This establishes the equivalence \eqref{eq:divisibility}, finishing 
the proof of Claim~\ref{cl:substructure-in-Lor-plus}.
\end{proof}

Similarly, it can be proved $D/f^{-1}(p_2) \subseteq V_1$ in the
language $\Lor \cup \{\mid\}$.  

\medskip

We come back to the proof of the amalgamation property for
$(\rT_2)_\forall$ in Case~1. 

Return to $B = B_1 \x B_2$, where $B_1$ and $B_2$ are real closed
valuation rings. Considering now the monomorphism $g \colon D \to B$,
one can show, by the previous arguments for~$f$, that
$$
\specm(D) = \bigl\{ g^{-1}(q_1), g^{-1}(q_2) \bigr\},
$$
where $q_1$ and $q_2$ are the minimal prime ideals of~$B$, given by
$q_1 := \set{(a,0) : a \in B_1}$ and $q_2 := \set{(0,b) : b \in B_2}$.
So without loss of generality, it can be assumed that
$f^{-1}(p_1) = g^{-1}(q_1)$ and $f^{-1}(p_2) = g^{-1}(q_2)$.
It can be shown similarly that $D/g^{-1}(q_1) \hookto B_2$ and
$D/g^{-1}(q_2) \hookto B_1$ by monomorphisms in the language
$\Lor \cup \{\mid\}$. These monomorphisms are none other than
$g_2 := \pi_2 \circ g$ and $g_1 := \pi_1 \circ g$, respectively: 
$$
\begin{tikzcd}[column sep=small]
& V_1
\\
D_1 = D/f^{-1}(p_2) = D/g^{-1}(q_2) \qquad
\ar[dr, "g_1"'] \ar[ur, "f_1"]
\\
& B_1,
\end{tikzcd}
$$
and
$$
\begin{tikzcd}[column sep=small]
& V_2
\\
D_2 = D/f^{-1}(p_1) = D/g^{-1}(q_1) \qquad 
\ar[dr, "g_2"'] \ar[ur, "f_2"]
\\
& B_2.
\end{tikzcd}
$$
By the amalgamation property of the universal theory of RCVR, there
exist two real closed valuation rings $C_1$ and $C_2$, and
monomorphisms $h_i \colon V_i \to C_i$ and $k_i \colon B_i \to C_i$ in
the language $\Lor \cup \{\mid\}$, for $i = 1,2$; such that the
following diagrams are commutative:
\begin{equation}
\begin{tikzcd}
& V_1 \ar[dr, "h_1"]
\\
D_1 \ar[ur, "f_1"] \ar[dr, "g_1"'] & \conmutar & C_1
\\
& B_1 \ar[ur, "k_1"'] 
\end{tikzcd}
\word{and}
\begin{tikzcd}
& V_2 \ar[dr, "h_2"]
\\
D_2 \ar[ur, "f_2"] \ar[dr, "g_2"'] & \conmutar & C_2.
\\
& B_2 \ar[ur, "k_2"']  
\end{tikzcd}
\label{eq:two-diagrams} 
\end{equation}
These two diagrams can be glued together in the following one:
$$
\begin{tikzcd}[column sep=tiny]
\qquad \qquad
& D/f^{-1}(p_2) \x D/f^{-1}(p_1) \subseteq V_1 \x V_2  \ar[dr, "h"]
\\
D \ar[ur, "f"] \ar[dr, "g"'] & \conmutar & C_1 \x C_2,
\\
\qquad \qquad
& D/g^{-1}(q_2) \x D/g^{-1}(q_1) \subseteq B_1 \x B_2  \ar[ur, "k"']
\end{tikzcd}
$$
where $h = (h_1,h_2)$ and $k = (k_1,k_2)$. The last diagram does
commute, since for each $d \in D$,
\begin{align*}
(h \circ f)(d) = h(f(d))
&= (h_1, h_2) \circ \bigl( f_1(d), f_2(d) \bigr) 
\\
&= \bigl( (h_1 \circ f_1)(d), (h_2 \circ f_2)(d) \bigr)
\\
&= \bigl( (k_1 \circ g_1)(d), (k_2 \circ g_2)(d) \bigr)
\\
&= (k_1, k_2) \circ \bigl( g_1(d), g_2(d) \bigr) = (k \circ g)(d).  
\end{align*}
It is obvious that $C = C_1 \x C_2 \models \rT_2$.
Using the following equivalences:
\begin{align*}
(b_1,b_2) \rad (a_1,a_2)
&\iff b_1 \rad a_1 \mot{and} b_2 \rad a_2,
\\
(b_1,b_2) \mid (a_1,a_2)
&\iff b_1 \mid a_1 \mot{and} b_2 \mid a_2,
\\
(b_1,b_2) \divloc (a_1,a_2)
&\iff (a_1 = a_2 = 0) \mot{or} b_1 \mid a_1 \mot{or} b_2 \mid a_2,
\end{align*}
it follows easily that $h \colon V_1 \x V_2 \to C_1 \x C_2$ and 
$k \colon B_1 \x B_2 \to C_1 \x C_2$ are monomorphisms in the language
$\Lanre \cup \{\mid,\rad,\divloc\}$.
End of Case 1.

\paragraph{Case 2}
\textit{$D$ is a domain}.

The Boolean algebra of idempotents is now $\id D = \{0,1\}$. As in the
previous case, we denote $f \colon D \to A = V_1 \x V_2$ by
$f(d) = \bigl(f_1(d), f_2(d) \bigr)$, with $f_1$ and $f_2$ being ring
homomorphisms. Here, it is clear that $f_1$ and $f_2$ are actually
monomorphisms: for nonzero $x \in D$, $1 \rad x$ and therefore
$f(1) \rad f(x)$, i.e., $(1,1) \rad \bigl( f_1(x), f_2(x) \bigr)$, so
$f_1(x) \neq 0$ and $f_2(x) \neq 0$. It is easily verified that
$f_1$ and $f_2$ are also monomorphism of ordered rings, that is, in
the language~$\Lor$. 

To see that $f_1$ and $f_2$ preserve divisibility, one must use that
$f$ preserves the local divisibility. If $d_1 \mid d_2$ in~$D$, then
$f(d_1) \mid f(d_2)$ in~$A$. Therefore $f_1(d_1) \mid f_1(d_2)$
in~$V_1$ and also $f_2(d_1) \mid f_2(d_2)$ in~$V_2$.

For the reverse implication, suppose that $d_1, d_2 \in D$ entail
$f_1(d_1) \mid f_1(d_2)$ in~$V_1$. Then $f(d_1) \divloc f(d_2)$ and
therefore $d_1 \divloc d_2$ in~$D$. Since $D$ is a domain, this
entails $d_1 \mid d_2$ in~$D$. Similarly, $f_2$ preserves
divisibility.
 
\medskip
 
Therefore $f_1 \colon D \to V_1$ and $f_2 \colon D \to V_2$ are
monomorphisms in the language $\Lor \cup \{\mid\}$. And as before, 
there are monomorphisms $g_1 \colon D \to B_1$ and 
$g_2 \colon D \to B_2$ in the language $\Lor \cup \{\mid\}$, where
$g(d) = \bigl( g_1(d), g_2(d) \bigr)$ for all $d \in D$.

By the elimination of quantifiers of the theory RCVR and the
amalgamation property of its universal theory, one has the following
commutative diagrams, not unlike those of~\eqref{eq:two-diagrams}:
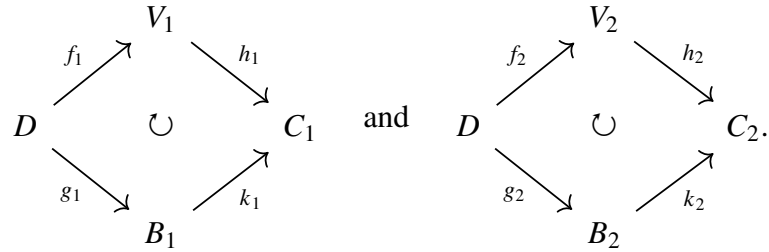
$$
\begin{tikzcd}
& V_1 \ar[dr, "h_1"]
\\
D \ar[ur, "f_1"] \ar[dr, "g_1"'] & \conmutar & C_1
\\
& B_1 \ar[ur, "k_1"'] 
\end{tikzcd}
\word{and}
\begin{tikzcd}
& V_2 \ar[dr, "h_2"]
\\
D \ar[ur, "f_2"] \ar[dr, "g_2"'] & \conmutar & C_2.
\\
& B_2 \ar[ur, "k_2"']  
\end{tikzcd}
$$
where again $C_1$ and $C_2$ are real closed valuation rings and
$h_1$, $h_2$, $k_1$, $k_2$ are monomorphisms of ordered rings
preserving divisibility. Considering $C = C_1 \x C_2 \models \rT_2$
and $h = (h_1,h_2)$, $k = (k_1,k_2)$, the following diagram is
completed:
$$
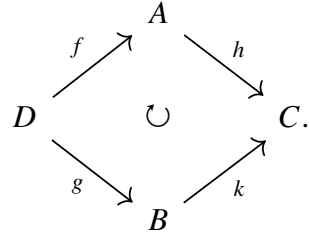
\begin{tikzcd}
& A \ar[dr, "h"]
\\
D \ar[ur, "f"] \ar[dr, "g"'] & \conmutar & C.
\\
& B \ar[ur, "k"'] 
\end{tikzcd}
$$
As before, it can be proved that $h$ and $k$ are monomorphisms in
$\Lanre \cup \{\mid,\rad,\divloc\}$, and 
the commutativity of this diagram can likewise be checked.
End of Case~2.

\medskip

The amalgamation property for the universal theory of $\rT_2$ has now
been established in both cases, thereby completing the proof of
Lemma~\ref{lm:T2-amalgamates}.
\end{proof}

\begin{proof}[Proof of Theorem~\ref{th:T2-is-qe}]
The elimination of quantifiers for $\rT_2$ now follows directly from
Lemmas \ref{lm:T2-is-model-complete} and~\ref{lm:T2-amalgamates}.
\end{proof}

Now consider $\rF_2 = \Th(F)$ where $F = F_1 \x F_2$ with $F_1$
and~$F_2$ being two real closed fields. In fact, $\rF_2$ is the theory
of von~Neumann regular real closed rings where the are only four
idempotents. Proceeding in the same way as the previous
theory~$\rT_2$, one can prove a similar result.

\begin{prop} 
\label{pr:F2-is-qe}
$\rF_2$ has elimination of quantifiers in the language
$\Lanre \cup \{\rad\}$.
\qed
\end{prop}

It is of interest to see when the divisibility relation can be viewed
as a radical relation.

\begin{lem} 
\label{lm:when-divisibility-is-radical}
Let $A$ be a ring. The divisibility defines a radical relation
$a \rad b$ on~$A$ by $b \mid a$, if and only if $A$ is a von~Neumann
regular ring. In that case, the divisibility is given by the radical
relation associated to the minimal prime spectrum.
\end{lem}

\begin{proof}
Set $a \rad b \iff b \mid a$, for all $a,b \in A$. This binary
relation clearly satisfies the conditions (1)~to~(5) defining a
radical relation, see Eq.~\eqref{eq:radical-relations} above. This
relation $\rad$ satisfies condition~(6) if and only if $b^2 \mid b$
for all $b \in A$. Hence $\rad$ satisfies condition~(6) if and only
if $A$~is a von~Neumann regular ring. In short, $a \rad b$ defined by
$b \mid a$ is a radical relation on~$A$ if and only if $A$ is a
von~Neumann regular ring.

\medskip

Suppose that $A$ is a von Neumann regular ring. We claim that
$b \mid a$ if and only if $\forall x\,(bx = 0 \implies ax = 0)$.  
First assume there is some $c \in A$ with $bc = a$; if $x \in A$
satisfies $bx = 0$, then clearly $(bx)c = 0$, and then 
$ax = (bc)x = 0$. On the other hand, suppose that 
$\forall x\,(bx = 0 \implies ax = 0)$ with $a,b \in A$. By the
von~Neumann regularity of~$A$, we can find $b^* \in A$ such that
$b^2 b^* = b$. Clearly $b(1 - bb^*) = b - b^2 b^* = 0$, and by our
hypothesis we get $a(1 - bb^*) = 0$, that is $a = a(bb^*)$, and then
$b(ab^*) = a$; so $b \mid a$ in $A$. We have shown that the
divisibility is the radical relation associated to the minimal prime
spectrum, see Eq.~\eqref{eq:radical-relation-pi-A}.
\end{proof}

\begin{prop} 
\label{pr:F2-is-still-qe}
$\rF_2$ has elimination of quantifiers in the language
$\Lanre \cup \{\mid,\rad,\divloc\}$. 
\end{prop}

\begin{proof}
For $\rF_2$, Lemma~\ref{lm:when-divisibility-is-radical} shows that
divisibility is equivalent to the radical relation. Moreover, from the
end of the proof of Theorem~\ref{th:T2-is-qe}, the local divisibility
is given by
$$
(b_1,b_2) \divloc (a_1,a_2) 
\iff (a_1 = a_2 = 0) \mot{or} (b_1 \mid a_1) \mot{or} (b_2 \mid a_2),
$$
or equivalently:
\begin{align*}
(b_1,b_2) \divloc (a_1,a_2) 
& \iff (a_1,a_2) = (0,0) \mot{or} (a_1 \rad b_1) 
\mot{or} (a_2 \rad b_2)
\\
& \iff (a_1,a_2) = (0,0) \mot{or} (a_1,0) \rad (b_1,0)
\mot{or} (0,a_2) \rad (0,b_2).
\end{align*}
Hence, the local divisibility is defined (without quantifiers) by the
radical relation.
\end{proof}

\section{The main theorem} 
\label{sec:punchline}

In this section, we shall interpret one language into another and 
vice~versa. It helps to adopt a simpler notation:
\begin{equation}
\Le_1 = \Lanre \cup \{\mid,\rad,\divloc\} \word{and}
\Le_2 = \Lanre \cup \{\div(\cdot,\cdot)\}.
\label{eq:less-jargon} 
\end{equation}

\begin{rem} 
\label{rk:div-is-L1-definable}
Let $A$ be a divisible-projectable $f$-ring. The symbol
$\div(\cdot,\cdot)$ may be defined (with quantifiers) modulo
$\Th(A)$, in the language $\Le_1$. Note that the negation of local
divisibility is:
$$
y \notdivloc w \iff w \neq 0 \land \forall w' \bigl( 
w' \neq 0 \land w'(w - w') = 0 \to y \nmid w' \bigr).
$$
Therefore,
$$
(y \notdivloc w) \lor w = 0 \iff \forall w' \bigl(
w' \neq 0 \land w'(w - w') = 0 \to y \nmid w' \bigr).
$$
Then, recalling~\eqref{eq:div-rigmarole}:
\begin{align*}
\div(x,y) = c 
& \longequi c \rad y \land \exists z \exists w \bigl(
x = z + w \land z \cdot w = 0 \land cy = z
\land (y \notdivloc w \lor w = 0 ) \bigr)
\\
& \longequi c \rad y \land \exists z \exists w \bigl(
(x = z + w \land z \cdot w = 0 \land cy = z 
\land y \notdivloc w ) \lor cy = x \bigr).
\end{align*}
Remark that when $w = 0$ then $z = x$ and $cy = x$.
\end{rem}

\begin{rem} 
\label{rk:rad-is-L2-definable}
Let $A$ be a reduced projectable and divisible-projectable $f$-ring.
The radical relation is defined (without quantifiers) in the language
$\Le_2$. Then, by Eq.~\eqref{eq:three-faces-of-A-two}:
$$
A \in \Ga^a_{\Lor\cup\{\mid\}} \bigl( X, (A_x)_{x\in X} \bigl),
$$
where $X$ is a Boolean space and $(A_x)_{x \in X}$ is a family of
totally ordered domains. Clearly,
\begin{align*} 
a \rad b &\iff \forall x \in X\,
\bigl( b(x) = 0 \implies a(x) = 0 \bigr)
\\
&\iff \booleana{b = 0} \subseteq \booleana{a = 0}
\\
&\iff \booleana{a \neq 0} \subseteq \booleana{b \neq 0}
\\
&\iff \supp(a) \subseteq \supp(b). 
\end{align*}
According to \eqref{eq:div-values}
and~\eqref{eq:three-faces-of-A-three} above, it follows that
$$
\div(a,b) = \begin{cases}
c & \text{if } b \neq 0 \land bc = a, \\
0 & \text{if } b = 0 \lor (b \neq 0 \land b \nmid a), \end{cases}
$$
in each $A_i$, for $i \in X$. Therefore,
$$
\div(a,a) = \begin{cases}
1 & \text{if } a \neq 0 \land a \cdot 1 = a, \\
0 & \text{if } a = 0. \end{cases}
$$
Consequently, $\div(a,a)$ is the idempotent of~$A$ that represents
$\supp(a)$. Similarly for $b \in A$; and therefore:
$$
a \rad b \iff \div(a,a) \leq \div(b,b).
$$
\end{rem}

\begin{rem} 
\label{rk:div-is-L2-definable}
The divisibility relation is defined (without quantifiers) in the
language $\Le_2$. As in Remark~\ref{rk:rad-is-L2-definable}, let
$A$ be a reduced projectable and divisible-projectable $f$-ring
and let, by~\eqref{eq:three-faces-of-A-two}:
$$
A \in \Ga^a_{\Lor\cup\{\mid\}} \bigl( X, (A_x)_{x\in X} \bigl),
$$
where $X$ is a Boolean space and $\set{A_x : x \in X}$ is a family of
totally ordered domains. We assert:
$$
b \mid a \iff \forall x \in X \bigl(
a(x) \neq 0 \implies \div (a,b) (x) \neq 0 \bigr).
$$ 
($\Rightarrow$): Assume $b \mid a$ and $a(x) \neq 0$. Then there is
$c \in A$ with $bc = a$; therefore $b(x)c(x) = a(x) \neq 0$. Thus,
$c(x) \neq 0$ and $b(x) \neq 0$.  Clearly, by \eqref{eq:div-values}
and~\eqref{eq:three-faces-of-A-three}, 
$$
\div_A(a,b)(x) = \div_{A_x}(a(x), b(x)) = c(x) \neq 0.  
$$
($\Leftarrow$): Assume instead that $\forall x \in X\,
\bigl( a(x) \neq 0 \implies \div(a,b)(x) \neq 0 \bigr)$. If $a = 0$,
clearly $b \mid a$. If $a \neq 0$, let 
$N := \booleana{a \neq 0} \neq \emptyset$. Let $c = \div(a,b)$, so
$c \rad b$ and there are $z,w \in A$ such that $a = z + w$, $zw = 0$,
$bc = z$ and $\forall w'\,
\bigl( w' \neq 0 \land w'(w - w') = 0 \to b \nmid w' \bigr)$. For
$x \in N$, by hypothesis $\div(a,b)(x) = c(x) \neq 0$ holds. Now
$c \rad b$ entails $b(x) \neq 0$, and so $b(x) c(x) = z(x) \neq 0$.
Thus $w(x) = 0$ and $a(x) = z(x)$, or equivalently
$b(x) c(x) = a(x)$. If $x \notin N$, then $a(x) = 0$ and
$b(x) \cdot 0 = a(x)$. Redefining 
$d = c_{\rest_N} \cup 0_{\rest_{X\setminus N}} \in A$, one obtains
$bd = a$; and thus $b \mid a$.

\medskip  

We have now shown that
$$
b \mid a \iff \supp(a) \subseteq \supp(\div(a,b))
\iff a \rad \div (a,b).
$$
By Remark~\ref{rk:rad-is-L2-definable}, we deduce that
$$
b \mid a \iff \div(a,a) \leq \div\bigl( \div(a,b), \div(a,b) \bigr).
$$
\end{rem}

In, \cite[Prop.~4.14]{Guier4}, the following formula was proved for a
projectable reduced $f$-ring:
$$
\forall a \forall b \forall c\,
\bigl( a \notrad bc - a \longrightarrow b \divloc a \bigr).
$$
This formula is equivalent to
\begin{equation}
\forall a \forall b\, \bigl(
\exists c\, (a \notrad bc - a) \longrightarrow b \divloc a \bigr).
\label{eq:divloc-formula} 
\end{equation}
Note that if $a = 0$ in the hypothesis of the previous formula, then
$0 \notrad bc$; but this is always false since $0 \rad x$ for all
$x \in A$. So the hypothesis of~\eqref{eq:divloc-formula}, avoids
$a = 0$. This is the only restriction to deduce the reverse
implication, as the following result confirms.

\begin{prop} 
\label{pr:local-divisibility}
Let $A$ be a projectable reduced $f$-ring. Then the formula
\begin{equation}
\forall a \forall b\, \bigl( \exists c\,(a \notrad bc - a)
\longequi (b \divloc a \land a \neq 0) \bigr)
\label{eq:local-divisibility2} 
\end{equation}
is valid in~$A$.
\end{prop}

\begin{proof}
The implication ($\Rightarrow$) was already discussed. For the
converse, invoking~\eqref{eq:class-of-A}, let
$$
A \in \Ga^a_{\Lor} \bigl( X, (A_x)_{x\in X} \bigl),
$$
where $X$ is a Boolean space and $\set{A_x : x \in X}$ is a family of
totally ordered domains. 

Suppose that $b \divloc a$ and $a \neq 0$. There is $w \in A$ such
that $w \neq 0$, $w(w - a) = 0$ and $bc = w$ for some $c \in A$. Since
$w \neq 0$, there exists $x_0 \in X$ such that $w(x_0) \neq 0$; and
then $w(x_0) = a(x_0) \neq 0$. Therefore
$b(x_0)c(x_0) = w(x_0) = a(x_0)$. In summary, $x_0 \in X$ satisfies
$a(x_0) \neq 0$ and $(bc - a)(x_0) = 0$. That entails
$a \notrad bc - a$, according to the definition 
\eqref{eq:radical-relation-pi-A} of $\rad$ for the minimal prime
spectrum. And we have shown that $\exists c\,(a \notrad bc - a)$.
\end{proof}

It is helpful to express the preceding formula
\eqref{eq:local-divisibility2} in the following contrapositive form:
$$
\forall a \forall b\, \bigl( \forall c\,(a \rad bc - a) 
\longequi (b \notdivloc a \lor a = 0) \bigr).
$$

\begin{rem} 
\label{rk:divloc-is-L2-definable}
The local divisibility relation is defined (without quantifiers) in
the language~$\Le_2$. Let $A$ be a reduced projectable and
divisible-projectable $f$-ring. We assert that for $a,b \in A$,
\begin{equation}
b \divloc a  \iff  a = 0 \lor \bigl( b \cdot \div(a,b) \neq 0 \bigr).
\label{eq:divloc-sans-quantifiers} 
\end{equation}

To show this, we reiterate \eqref{eq:three-faces-of-A-two}:
$$
A \in \Ga^a_{\Lor\cup\{\mid\}} \bigl( X, (A_x)_{x\in X} \bigl),
$$
where $X$ is a Boolean space and $\set{A_x : x \in X}$ is a family of
totally ordered domains.  

\medskip

($\Rightarrow$):
Suppose that $b \divloc a$ and $a \neq 0$. There is some $w \in A$
such that $w \neq 0$, $w(w - a) = 0$ and $b \mid w$. Since $w \neq 0$,
there exists $x \in X$ such that $w(x) \neq 0$. As usual,
$w(x) = a(x)$ and $b(x) \mid w(x) = a(x)$. There exists $c_x \in A_x$
such that $b(x) \cdot c_x = a(x)$, with $a(x) \neq 0$ and
$b(x) \neq 0$; set $\div (a,b) = c$. By the description of the
$\div$~symbol in the theory of totally ordered domains in
\cite[Lem.~4.4]{Guier1}, see~\eqref{eq:div-values}, 
$\div(a,b)(x) = c(x) = c_x \neq 0$. Then
$b(x) \cdot \div(a,b)(x) = b(x) \cdot c_x = a(x) \neq 0$. Hence
$b \cdot \div (a,b) \neq 0$.

\medskip

($\Leftarrow$):
If $a = 0$, then clearly $b \divloc a$ 
by~\eqref{eq:local-divisibility}. If $a \neq 0$, suppose
$b \cdot \div(a,b) \neq 0$. Set $c = \div(a,b)$. Then there is an
$x \in X$ such that $b(x) \neq 0$ and $c(x) \neq 0$. By the definition
\eqref{eq:div-rigmarole} of the $\div$ symbol, there are elements
$z,w \in A$ such that $a = z + w$, $z \cdot w = 0$, $bc = z$ and
$\forall w'\,(w' \neq 0 \land w'(w - w') = 0 \to b \nmid w')$. Then
$b(x)c(x) = z(x) \neq 0$ and therefore $w(x) = 0$ and
$a(x) = z(x) \neq 0$; we have shown the existence of $x \in X$ such
that $(bc - a)(x) = 0$ and $a(x) \neq 0$. In short,
$a \notrad bc - a$. Then the formula $\exists c\,(a \notrad bc - a)$
is valid in~$A$. By the previous
Proposition~\ref{pr:local-divisibility}, or \cite[Prop.~4.14]{Guier4},
we conclude that $b \divloc a$.
\end{rem}

We sum up the foregoing propositions in the following theorem.

\begin{them} 
\label{th:punchline}
Let $A$ be a reduced projectable and divisible-projectable $f$-ring
satisfying the first convexity property. Then $\Th(A)$ admits
elimination of quantifiers in $\Lanre \cup \{\mid,\rad,\divloc\}$ if
and only if one of the following conditions holds:
\begin{enumerate}[noitemsep, label=\textup{(\roman*)}]
\item 
$A$ is a real closed field,
\item 
$A$ is a product of two real closed fields,
\item 
$A$ is a von Neumann regular real closed ring without nonzero minimal
idempotents,
\item 
$A$ is a real closed valuation ring,
\item 
$A$ is a product of two real closed valuation rings,
\item 
$A$ is a sc-regular real closed ring without nonzero minimal
idempotents.
\end{enumerate}
\end{them}

\begin{proof}
Recall Eq.~\eqref{eq:less-jargon}:
$\Le_1 = \Lanre \cup \{\mid,\rad,\divloc\}$ and
$\Le_2 = \Lanre \cup \{\div(\cdot,\cdot)\}$.

\medskip

($\Rightarrow$):
By Remark~\ref{rk:div-is-L1-definable}, any formula in~$\Le_2$ is
equivalent to a formula in~$\Le_1$, possibly by adjoining more
quantifiers. The hypothesis says that this $\Le_1$-formula is
equivalent to a quantifier-free $\Le_1$-formula. By Remarks
\ref{rk:rad-is-L2-definable}, \ref{rk:div-is-L2-definable} and
\ref{rk:divloc-is-L2-definable}, the latter formula is equivalent to a
quantifier-free $\Le_2$-formula. Thus, $\Th(A)$ admits elimination of
quantifiers in~$\Le_2$. The result then follows by
\cite[Thm.~4.10]{Guier1}. This direction is essentially
\cite[Thm.~4.10]{Guier1} in the language~$\Le_2$. Indeed, what we have
done in this Section, cf.\ Remarks \ref{rk:rad-is-L2-definable},
\ref{rk:div-is-L2-definable} and~\ref{rk:divloc-is-L2-definable}, was
to prove that the symbols of~$\Le_1$ are definable without quantifiers
in the language~$\Le_2$ -- using the $\div(\cdot,\cdot)$ binary
symbol. The proof of the `only~if' part of \cite[Thm.~4.10]{Guier1}
rests on previous results on discriminator varieties by
F.~Point~\cite{Point-vd}.

\medskip

($\Leftarrow$): For the converse:
\begin{itemize}[nosep]
\item
Case~(i) is Tarski's theorem \cite{Tarski1, Tarski2}.
\item
Case~(ii) is given by Proposition~\ref{pr:F2-is-still-qe}. 
\item
Case~(iii) is the last Theorem on page~27 in~\cite{Pre-Schw}.
\item
Case~(iv) is the second section of~\cite{Ch-Dick2}.
\item
Case~(v) is given by Theorem~\ref{th:T2-is-qe}.
\item 
Case~(vi) is \cite[Thm.~6.13]{Guier4}.
\qed
\end{itemize}
\hideqed     
\end{proof}

This theorem has some advantages over \cite[Thm.~4.10]{Guier1}.
Firstly, the language $\Le_1$ 
has a better mathematical sense and it is more efficient than~$\Le_2$,
since the binary function symbol $\div(\cdot,\cdot)$ was replaced by
(binary) relations. Secondly, considering that the definition of the
$\div(\cdot,\cdot)$ is restricted to the divisible-projectability
condition on the ring, and that the local-divisibility relation is
not, we may ask whether there exists a larger class of real closed
rings admitting quantifier elimination in the language
$\Le_1 = \Lanre \cup \{\mid,\rad,\divloc\}$.

\subsection*{Acknowledgements}

I am grateful to the referee for corrections and suggestions that
improved the introduction and the presentation of the paper. Many
thanks to Michael Josephy and Joseph C. Várilly, who recommended
several improvements to the English (and Latex) style. I acknowledge
financial support from the Vicerrectoría de Investigación of the
University of Costa Rica via the research project 821--C3--192.


\end{document}

%% file: fonts-charlas7.tex
\DeclareMathOperator{\Ann}{Ann} 
\renewcommand{\div}{\operatorname{div}}

\DeclareMathOperator{\im}{im}

\DeclareMathOperator{\spec}{Spec}   
\DeclareMathOperator{\specm}{Specmin}   
\DeclareMathOperator{\supp}{supp} 
\DeclareMathOperator{\Th}{Th}   

\newcommand{\Ga}{\Gamma}          

\newcommand{\sI}{\mathcal{I}}

\newcommand{\sL}{\mathcal{L}}
\newcommand{\Le}{\mathcal{L}}

\newcommand{\loc}{{\mathrm{loc}}}

\newcommand{\rF}{\mathrm{F}}      
\newcommand{\rT}{\mathrm{T}}      

\newcommand{\conmutar}{\circlearrowright}

\newcommand{\ett}{\mathbin{\&}}     
         
\newcommand{\hookto}{\hookrightarrow}  

\renewcommand{\leq}{\leqslant}      
\newcommand{\longequi}{\longleftrightarrow}

\newcommand{\rad}{\preceq}
\newcommand{\notrad}{\not\preceq}
\newcommand{\rest}{\upharpoonright}
\newcommand{\selem}{\prec}

\newcommand{\w}{\wedge}             
\newcommand{\x}{\times}             

\newcommand{\lcro}{\llbracket}       
\newcommand{\rcro}{\rrbracket}       

\newcommand{\divloc}{\mid_{\loc}}

\newcommand{\Lanre}{\sL_{\mathrm{lor}}}
\newcommand{\locdiv}{\mid_{\loc}}
\newcommand{\Lor}{\sL_{\mathrm{or}}}
\newcommand{\notdivloc}{\nmid_{\loc}}

\newcommand{\bipol}[1]{{#1}^{\perp\perp}}
\newcommand{\booleana}[1]{\lcro#1\rcro}      
\newcommand{\id}[1]{\sI d(#1)}
\newcommand{\mot}[1]{\enspace\text{#1}\enspace} 

\newcommand{\set}[1]{\{\,#1\,\}}             
\newcommand{\Set}[1]{\bigl\{\,#1\,\bigr\}}   
\newcommand{\word}[1]{\quad\text{#1}\quad}   

